\documentclass[letterpaper,11pt]{amsart}
\usepackage{amsmath,amssymb,amsxtra,amsthm, amstext, amscd,amsfonts,fancyhdr, hyperref}
\newcommand{\bC}{{\mathbb{C}}}

\newcommand{\bF}{{\mathbb{F}}}

\newcommand{\bP}{{\mathbb{P}}}
\newcommand{\bQ}{{\mathbb{Q}}}
\newcommand{\bR}{{\mathbb{R}}}

\newcommand{\bZ}{{\mathbb{Z}}}

  \newcommand{\A}{{\mathcal{A}}}
  \newcommand{\B}{{\mathcal{B}}}
  \newcommand{\C}{{\mathcal{C}}}
  \newcommand{\D}{{\mathcal{D}}}

  \newcommand{\G}{{\mathcal{G}}}

  \newcommand{\J}{{\mathcal{J}}}

  \newcommand{\M}{{\mathcal{M}}}
  \newcommand{\N}{{\mathcal{N}}}

\renewcommand{\S}{{\mathcal{S}}}
  \newcommand{\T}{{\mathcal{T}}}
  \newcommand{\U}{{\mathcal{U}}}
  \newcommand{\V}{{\mathcal{V}}}

\newcommand{\Gal}{\operatorname{Gal}}
\newcommand{\GL}{\operatorname{GL}}
\newcommand{\PGL}{\operatorname{PGL}}
\newcommand{\Aut}{\operatorname{Aut}}

\newcommand{\fC}{\mathfrak{C}}

\newcommand{\ep}{\varepsilon}

\newcommand{\ol}{\overline}

\newcommand{\upchi}{{\raise.35ex\hbox{$\chi$}}}

\makeatletter
\@namedef{subjclassname@2010}{%
  \textup{2010} Mathematics Subject Classification}
\makeatother

\newtheorem{theorem}{Theorem}[section]

\newtheorem{proposition}[theorem]{Proposition}
\newtheorem{lemma}[theorem]{Lemma}

\theoremstyle{definition}

\numberwithin{equation}{section}

\begin{document}

\title{On binary cubic and quartic forms}
\author{Stanley Yao Xiao}
\address{University of Oxford, Mathematical Institute, Andrew Wiles Building,
Radcliffe Observatory Quarter,
Woodstock Road,
Oxford, UK OX2 6GG}
\email{stanley.xiao@maths.ox.ac.uk}
\date{\today}


\begin{abstract} In this paper we determine the group of rational automorphisms of binary cubic and quartic forms with integer coefficients and non-zero discriminant in terms of certain quadratic covariants of cubic and quartic forms. This allows one to give precise asymptotic formulae for the number of integers in an interval representable by a binary cubic or quartic form and extends work of Hooley. Further, we give the field of definition of lines contained in certain cubic and quartic surfaces related to binary cubic and quartic forms. 
\end{abstract}

\maketitle

\section{Introduction}
\label{S1}

Let $F$ be a binary form with integer coefficients, non-zero discriminant, and degree $d \geq 3$. For each positive number $Z$, put $R_F(Z)$ for the number of integers of absolute value at most $Z$ which is representable by the binary form $F$. In \cite{Hoo1}, \cite{Hoo1-2}, and \cite{Hoo2} C.~Hooley gave explicitly the asymptotic formula for the quantity $R_F(Z)$ when $F$ is an irreducible binary cubic form or a biquadratic quartic form. Various authors have dealt with the case when $F$ is a diagonal form; see \cite{SX} for a summary of these results.\\ \\
In \cite{SX}, Stewart and Xiao proved the existence of an asymptotic formula for $R_F(Z)$ for all $F$ with integer coefficients, non-zero discriminant, and degree $d \geq 3$. More precisely, they proved that for each such $F$, there exists a positive rational number $W_F$ which depends only on $F$ for which the asymptotic formula
\begin{equation} \label{RFZ} R_F(Z) \sim W_F A_F Z^{\frac{2}{d}}  \end{equation}
holds with a power-saving error term. Here $A_F$ is the area of the region
\[\{(x,y) \in \bR^2 : |F(x,y)| \leq 1\}.\]
The power-saving error term is obtained from the $p$-adic determinant method developed by Heath-Brown in \cite{HB1}, and its subsequent refinement by Salberger \cite{S2}. \\ 

Stewart and Xiao showed that $W_F$ is an explicit function of the \emph{rational automorphism group} of $F$. To define this group, consider the substitution action of $\GL_2(\bC)$ on binary forms given as follows: for $T = \left(\begin{smallmatrix} t_1 & t_2 \\ t_3 & t_4 \end{smallmatrix} \right) \in \GL_2(\bC)$, put
\begin{equation} \label{sub action} F_T(x,y) = F(t_1 x + t_2 y, t_3 x + t_4 y).\end{equation}
Then the rational automorphism group $\Aut F$ of a binary form $F$ is defined to be:
\begin{equation} \label{Aut} \Aut F = \left\{ T  \in \GL_2(\bQ) : F_T(x,y) = F(x,y) \right \}.\end{equation}
We shall also denote by $\Aut_\bF F$ to be the maximal subgroup of $\GL_2(\bF)$ which fixes $F$ via the action (\ref{sub action}), for any subring $\bF$ of the complex numbers $\bC$. \\ 

In \cite{SX}, it was not shown how to obtain $\Aut F$, and therefore $W_F$, explicitly from the coefficients of $F$, except for the case of diagonal forms of the shape $Ax^d + By^d$. In general, this can be quite difficult. We shall show in this paper how to determine $W_F$ when $d = 3,4$, and thereby completing the work of Hooley in \cite{Hoo1}, \cite{Hoo1-2}, and \cite{Hoo2} for degrees 3 and 4. The goal of this paper is to prove the following: 

\begin{theorem} \label{MT1} Let $F$ be a binary form with integer coefficients, non-zero discriminant, and degree $d \in \{3,4\}$. Then for each $d = 3,4$ there exists a positive number $\beta_d < 2/d$ which depends only on $d$ and a positive rational number $W_F$ such that for all $\ep > 0$, the asymptotic formula
\begin{equation} \label{RFZ2} R_F(Z) = W_F A_F Z^{\frac{2}{d}} + O_{F,\ep} \left(Z^{\beta_d + \ep}\right)
\end{equation} 
holds. Moreover, the quantity $W_F$ can be explicitly determined in terms of the coefficients of $F$. 
\end{theorem}

Theorem \ref{MT1} will follow from Theorem 1.2 in \cite{SX} and Theorems \ref{BCFMT} and \ref{BQFMT}, which give explicit methods to determine $\Aut F$ from the coefficients of $F$ when $\deg F = 3,4$ respectively. \\ \\
It remains to give an explicit expression for the quantity $A_F$ in terms of the coefficients of $F$. In the cubic case this was done by Hooley himself in \cite{Hoo1} and \cite{Hoo1-2}, where he showed that $A_F$ is a constant times a power of the discriminant of $F$, and that this constant is expressible in terms of values of the gamma function. Bean gave an explicit formula for $A_F$ in terms of hypergeometric functions in \cite{BE}. We further note that for all $d \geq 3$ (not necessarily $d \in \{3,4\}$), $A_F$ was determined by Stewart and Xiao in \cite{SX} in the case of diagonal forms of the shape $Ax^d + By^d$, where they gave $A_F$ in terms of $\Delta(F)$ and the gamma function. \\

To obtain an explicit expression for $\Aut F$, we actually achieve slightly more in that we determine $\Aut_\bF F$ over any field $\bF$ containing the field of definition of $F$, including the complex numbers $\bC$. Our explicit characterization of automorphism groups of binary cubic and quartic forms allows us to study lines on algebraic surfaces of the shape 
\begin{equation} \label{surface def} X_F : F(x_1, x_2) - F(x_3, x_4) = 0.\end{equation}

For cubic surfaces, it is a celebrated theorem of Cayley and Salmon that cubic surfaces contain exactly 27 lines over an algebraically closed field. However, for a cubic surface defined over $\bQ$, these lines are typically not defined over $\bQ$. There exists a unique smallest finite extension $K/\bQ$ such that all $27$ lines are defined. In particular, for the generic cubic surface defined over $\bQ$, this field is Galois over $\bQ$ and its Galois group is isomorphic to $W(E_6)$, the Weyl group for the $E_6$ root system. Ekedahl \cite{Eke} found an explicit example of a cubic surface which realizes this bound. We shall prove that when $F$ is a cubic form the field of definition of the lines on the surface $X_F$ given by (\ref{surface def}) is very small. \\

For quartic surfaces, it is not known in general how many lines they contain. The generic quartic surface contains no lines; see \cite{BSa}. Recall that the $\PGL_2(\bC)$-automorphism group of a binary form $F$ is the maximal subgroup of $\PGL_2(\bC)$ which permutes the projective roots of $F$ via action by Mobius transformation. It is a consequence of Theorem 3.1 in \cite{BSa} that the surface $X_F$ given in (\ref{surface def}) contains exactly $d(d + \upsilon_F)$ many lines, where $\upsilon_F$ is the number of elements in the $\PGL_2(\bC)$-automorphism group of $F$ and $d$ is the degree of $F$. For the quartic case this was already known to Segre; see \cite{Seg}. \\ 

Put
\[F(x,y) = a_4 x^4 + a_3 x^3 y + a_2 x^2 y^2 + a_1 xy^3 + a_0 y^4.\]
The ring of polynomial invariants of binary quartic forms with respect to the action (\ref{sub action}) is generated by two elements, usually denoted by $I(F)$ and $J(F)$, given by
\begin{equation} \label{I} I(F) = 12 a_4 a_0 - 3 a_3 a_1 + a_2^2 \end{equation}
and
\begin{equation} \label{J} J(F) = 72 a_4 a_2 a_0 + 9 a_3 a_2 a_1 - 27 a_4 a_1^2 - 27 a_0 a_3^2 - 2 a_2^3.\end{equation}
It is known from the work of Klein \cite{Kle} and later Segre \cite{Seg} that the $\PGL_2(\bC)$-automorphism group of a binary quartic form $F$ with complex coefficients and non-zero discriminant is isomorphic to the Klein group $\C_2 \times \C_2$ unless the invariants $I(F), J(F)$ vanishes. Specifically, the $\PGL_2(\bC)$-automorphism group of a quartic form $F$ is isomorphic to the dihedral group $\D_4$ if $J(F) = 0$ and isomorphic to the alternating group $\A_4$ if $I(F) = 0$; see Proposition \ref{PGL auto prop}. We do not know how to explicitly determine the field of definitions of the lines on $X_F$ corresponding to the extra automorphisms when $I, J = 0$, but in the generic case when $I(F) \cdot J(F) \ne 0$, we can determine the field of definition of all lines on $X_F$. \\

We thus obtain the following theorem:

\begin{theorem} \label{line thm} Let $F$ be a binary cubic or quartic form with non-zero discriminant and integer coefficients. Let $X_F$ be the algebraic surface defined by (\ref{surface def}). Then 
\begin{itemize} 
\item[(a)] for $\deg F = 3$,  $X_F$ contains exactly $27$ distinct lines over $\ol{\bQ}$, and these lines are defined over a field of degree at most $12$ over $\bQ$.
\item[(b)] for $\deg F = 4$, $X_F$ contains exactly $32$ distinct lines over $\ol{\bQ}$ if both $I(F)$ and $J(F)$ are non-zero, $48$ lines when $J(F) = 0$, and $64$ lines when $I(F) = 0$. Further, when $I(F) \cdot J(F) \ne 0$, these lines are defined over a field of degree at most $48$ over $\bQ$.
\end{itemize}
\end{theorem}

We shall denote by $\C_n$ the cyclic group of order $n$, $\D_n$ for the dihedral group of order $2n$, $\A_n$ for the alternating group on $n$ letters, and $\S_n$ for the symmetric group on $n$ letters throughout this paper. Moreover, for a binary form $F$, we shall denote its discriminant by $\Delta(F)$.
 
\section{Automorphism groups of binary cubic and quartic forms over large fields}

For a binary form $F$ of degree $d$ with complex coefficients, let $\B_F$ denote the set of roots of $F$ in $\bP^1(\bC)$. An element $T = \left(\begin{smallmatrix} t_1 & t_2 \\ t_3 & t_4 \end{smallmatrix} \right) \in \PGL_2(\bC)$ acts on a point $\theta \in \bP^1(\bC)$ via the Mobius action
\begin{equation} \label{U action} T \theta = \frac{t_1 \theta + t_2}{t_3 \theta + t_4}.\end{equation}
For a finite set $\S \subset \bP^1(\bC)$, put
\[T \S = \{T \theta : \theta \in \S\}.\]
Define the $\PGL_2(\bC)$-\emph{automorphism group of }$F$ to be
\begin{equation} \label{PGL Aut} \Aut_\bC^\ast F = \{T \in \PGL_2(\bC) : T \B_F = \B_F\}. \end{equation}
It is easily seen that
\begin{equation} \Aut_\bC^\ast F = \{T \in \PGL_2(\bC) : F_T = \lambda F \text{ for some } \lambda \in \bC^\times\}. \end{equation}
It is well-known that $\Aut_\bC^\ast (F)$ can be embedded into the symmetric group $\S_{d}$ via the action (\ref{U action}) of $\PGL_2$ on the roots of $F(x,1)$, viewed as elements in $\bP^1(\bC)$. Moreover, the natural homomorphism
\[\Aut_\bC (F) \rightarrow \Aut_\bC^\ast (F)\]
has kernel given by $\{\mu_d I_{2 \times 2} : \mu_d \text{ is a } d \text{-th root of unity} \}$.   \\

We shall prove the following Proposition:

\begin{proposition} \label{PGL auto prop} Let $F$ be a binary form with non-zero discriminant, complex coefficients, and $\deg F \in \{3,4\}$. Then 
\[\Aut_\bC^\ast (F) \cong \begin{cases} \S_3, &\text{if } d = 3 \\
\C_2 \times \C_2, \D_4, \A_4, & \text{if } d = 4. \end{cases} \]
Moreover $\Aut_\bC^\ast(F) \cong \D_4$ when $d = 4$ if and only if $J(F) = 0$ and $\Aut_\bC^\ast(F) \cong \A_4$ if and only if $I(F) = 0$. 
\end{proposition}

We have the following lemma for binary cubic forms with non-zero discriminant:
\begin{lemma} \label{cubic conj} Let $F$ be a binary cubic form with complex coefficients and non-zero discriminant. Then $F$ is $\GL_2(\bC)$ equivalent to $xy (x+y)$. 
\end{lemma}

\begin{proof} This follows from the fact that the action of $\PGL_2(\bC)$ on $\bP^1(\bC)$ is 3-transitive. 
\end{proof} 

A similar lemma, due to Cayley (see \cite{Olv} for an account), holds in the quartic case:

\begin{lemma} \label{quartic conj} Let $F$ be a binary quartic form with complex coefficients and non-zero discriminant. Then there exists a complex number $A$ such that $F$ is $\GL_2(\bC)$ equivalent to $x^4 + A x^2 y^2 + y^4$. Moreover, every binary quartic form $F$ with complex coefficients and non-zero discriminant with $I(F) = 0$ is $\GL_2(\bC)$-equivalent to 
\[x(x^3 + y^3)\]
and equivalent to 
\[ x^4 + y^4\]
if $J(F) = 0$.  
\end{lemma} 
We shall next require the following lemma, which follows from simple group theory:
\begin{lemma} \label{conj lem} Let $F$ be a binary form with complex coefficients. Then for $T \in \GL_2(\bC)$, we have
\[\Aut_\bC F_T = T^{-1} (\Aut_\bC F) T\]
and likewise
\[\Aut_\bC^\ast F_T = T^{-1} (\Aut_\bC^\ast F) T.\]
\end{lemma}
This allows us to prove the following:
\begin{lemma} \label{3 and 4} Let $F$ be a binary quartic form with complex coefficients and non-zero discriminant. Then $\Aut_\bC^\ast F$ contains an element of order $3$ if and only if $I(F) = 0$, and contains an element of order $4$ if and only if $J(F) = 0$. 
\end{lemma}

\begin{proof} Suppose that $\Aut_\bC^\ast F$ contains an element $T$ of order 3. By Lemma \ref{conj lem}, we may assume that 
\[T = \begin{pmatrix} 1 & 0 \\ 0 & \mu_3 \end{pmatrix},\]
where $\mu_3$ is a primitive third root of unity. It then follows that $F$ is of the shape
\[F(x,y) = Ax(x^3 + y^3)\]
for some complex number $A$, and from (\ref{I}) one immediately sees that $I(F)  = 0$. A similar argument shows that if $\Aut_\bC^\ast F$ contains an element of order 4, then $F$ is equivalent to $x^4 + y^4$ and thus $J(F) = 0$. \\ \\
For the converse, if $\Delta(F) \ne 0$ then Lemma \ref{quartic conj} implies that $\Aut_\bC^\ast F$ contains elements of order 3 and 4 when $I(F) = 0$ and $J(F)  =0$, respectively. 
\end{proof}

Now we may give a proof of Proposition \ref{PGL auto prop}. 

\begin{proof}[Proof of Proposition \ref{PGL auto prop}] Let $F$ be a binary cubic form with complex coefficients and non-zero discriminant. Then, by Lemma \ref{cubic conj}, it follows that $F$ is $\GL_2(\bC)$-equivalent to $F_0 = xy(x+y)$. A quick calculation reveals that
\[\Aut_\bC^\ast F_0 = \left\{I_{2 \times 2}, \begin{pmatrix} 0 & 1 \\ -1 & -1 \end{pmatrix}, \begin{pmatrix} -1 & -1 \\ 1 & 0 \end{pmatrix}, \begin{pmatrix} 0 & 1 \\ 1 & 0 \end{pmatrix}, \begin{pmatrix} 1 & 0 \\ -1 & -1 \end{pmatrix}, \begin{pmatrix} -1 & -1 \\ 0 & 1 \end{pmatrix}\right\}.\]
It is routine to check that $\Aut_\bC^\ast F_0$ is isomorphic to $\S_3$, and so Lemma \ref{conj lem} shows that $\Aut_\bC^\ast F \cong \S_3$, as desired. \\ \\
Now let $F$ be a binary quartic form with non-zero discriminant. By Lemma \ref{quartic conj} it follows that $F$ is $\GL_2(\bC)$-equivalent to $F_1 = x^4 + Ax^2 y^2 + y^4$ for some complex number $A$. It is easily checked that $\Aut_\bC^\ast F_1$ contains the set
\[\left\{I_{2 \times 2}, \begin{pmatrix} 1 & 0 \\ 0 & -1 \end{pmatrix}, \begin{pmatrix} 0 & 1 \\ -1 & 0 \end{pmatrix}, \begin{pmatrix} 0 & 1 \\ 1 & 0 \end{pmatrix}\right\},\]
which is a set of representatives of a group isomorphic to $\C_2 \times \C_2$ in $\PGL_2(\bC)$. \\ \\
Note that $\Aut_\bC^\ast (F)$ is a subgroup of the symmetric group $\S_4$. By Lemma \ref{3 and 4}, $\Aut_\bC^\ast (F)$ contains an element of order $3$ when $I(F) = 0$. The only subgroups of $\S_4$ containing both $\C_2 \times \C_2$ and an element of order 3 are the alternating group $\A_4$ and $\S_4$ itself. By Lemma \ref{3 and 4} we see that $\Aut_\bC^\ast F$ cannot equal $\S_4$, since otherwise $\Aut_\bC^\ast F$ contains an element of order 4, which implies that $J(F) = 0$; and since 
\[\Delta(F) = \frac{4I(F)^3 - J(F)^2}{27},\]
this contradicts the assumption that $F$ has non-zero discriminant. Similarly, the only subgroups of $\S_4$ which contain $\C_2 \times \C_2$ and an element of order 4 are $\D_4$ and $\S_4$ itself, and the latter contains an element of order 3; hence cannot be isomorphic to $\Aut_\bC^\ast F$ for $F$ with non-zero discriminant by Lemma \ref{3 and 4}. This completes the proof of Proposition \ref{PGL auto prop}. \end{proof}

\section{Binary cubic forms}
\label{cubics}

Suppose
\[F(x,y) = b_3 x^3 + b_2 x^2 y + b_1 xy^2 + b_0 y^3\]
is a binary cubic form with integer coefficients and non-zero discriminant. We shall assume, after applying a $\GL_2(\bZ)$-action if necessary, that $b_3 \ne 0$. It is known that there is a single \emph{rational} quadratic covariant of $F$, given by the Hessian $q_F(x,y) = Ax^2 + Bxy + Cy^2$, where $A,B,C$ are as below:
\begin{equation} \label{hess cov coef} A = b_2^2 - 3b_3 b_1, B = b_2 b_1 - 9b_3 b_0, C = b_1^2 - 3b_2 b_0. \end{equation} 
Put $D = B^2 - 4AC$. It is known that $D = -3 \Delta(F)$. \\

In his thesis, G.~Julia identified three additional \emph{irrational}, or algebraic, quadratic covariants which depend on the roots $\theta_1, \theta_2, \theta_3$ of $F(x,1)$ in \cite{Jul}. We shall write the Julia covariant with respect to a root $\theta$ of $F(x,1)$ as follows:
\begin{equation} \label{Julia cov} J_\theta(x,y) = h_2 x^2 + h_1 xy + h_0 y^2,\end{equation}
where
\begin{align*} & h_2 = 9 b_3^2 \theta^2 + 6 b_3 b_2 \theta + 6 b_3 b_1 - b_2^2, \\
& h_1 = 6 b_3 b_2 \theta^2 + 6(b_2^2 - b_3 b_1)\theta + 2 b_2 b_1, \\
& h_0 = 3 b_3 b_1 \theta^2 + 3(b_2 b_1 - 3 b_3 b_0) \theta + 2 b_1^2 - 3 b_2 b_0.
\end{align*}
Cremona showed that $h_2, h_1, h_0$ are algebraic integers in \cite{Cre}, in the discussion immediately following equation (11). Thus, whenever $\theta$ is rational, $J_\theta$ has rational integral coefficients. \\ \\
For a binary quadratic form $f(x,y) = ax^2 + bxy + cy^2$ with complex coefficients, define
\begin{equation} \label{mat 1} \M_f = \begin{pmatrix} b & 2c \\ -2a & -b \end{pmatrix} \end{equation}
and
\begin{equation} \label{Hooley matrix} \N_f = \frac{1}{2 \Delta(f)} \begin{pmatrix} b \sqrt{-3 \Delta(f)} - \Delta(f) & 2c \sqrt{-3 \Delta(f)} \\ -2a \sqrt{-3 \Delta(f)} & -b \sqrt{-3 \Delta(f)} - \Delta(f) \end{pmatrix}. \end{equation}
Here the square root of a complex number is the principal square root with non-negative real part and positive imaginary part if the real part vanishes. \\ \\
Define
\begin{equation} \label{Ttheta} \T_\theta = \frac{-1}{6 \Delta(F)} \M_{J_\theta} \M_{q_F}. \end{equation}
When $\theta$ is rational, the matrix $6\Delta(F) \T_\theta$ has integer entries since $\Delta(q_F) = -3\Delta(F)$ and $\Delta(J_\theta) = 12 \Delta(F)$. \\ 

We have the following theorem:

\begin{theorem} \label{BCFMT} Let $F$ be a binary cubic form with integer coefficients and non-zero discriminant. Then: 
\begin{enumerate}
\item $\Aut F = \{I_{2 \times 2}\}$ if and only if $F$ is irreducible and $\Delta(F)$ is not a square.
\item $\Aut F$ is generated  by $\N_{q_F} \in \GL_2(\bQ)$ and is isomorphic to $\C_3$ if and only if $F$ is irreducible and $\Delta(F)$ is a square. 
\item $\Aut F$ is generated by $\T_\theta$ for the unique rational root $\theta$ of $F(x,1)$ and is isomorphic to $\C_2$ if and only if $F$ has exactly one rational linear factor over $\bQ$, corresponding to the root $\theta$.
\item 
\[\Aut F = \{I_{2 \times 2}, \N_{q_F}, \N_{q_F}^2, \T_{\theta_1}, \T_{\theta_2}, \T_{\theta_3}\} \cong \D_3\]
if and only if $F$ splits completely over $\bQ$. 
\end{enumerate}
\end{theorem}

We shall prove the following result, from which Theorem \ref{BCFMT} will follow:

\begin{proposition} \label{Julia gen} Let $F$ be a binary cubic form with complex coefficients and non-zero discriminant. Suppose that the $x^3$-coefficient of $F$ is non-zero and let $\theta_1, \theta_2, \theta_3$ be the three distinct roots of $F(x,1)$. Then a set of representatives of $\Aut_\bC^\ast (F)$ in $\GL_2(\bC)$ is given by 
\[\left\{I_{2 \times 2}, \T_{\theta_1}, \T_{\theta_2}, \T_{\theta_3}, \N_{q_F}, \N_{q_F}^2 \right\}.\]
\end{proposition}

\begin{proof} By Lemma \ref{cubic conj} and the observation that $q_F$ and $\J_{\theta_i}, i = 1,2,3$ are covariants of $F$, it suffices to prove Proposition \ref{Julia gen} for any binary cubic form with non-zero discriminant and non-zero leading coefficient. We choose
\[F(x,y) = 2x^3 + 3x^2 y + xy^2.\]
The roots of $F(x,1)$ are then $\theta_1 = 0, \theta_2 = -1, \theta_3 = -1/2$. Computing the Julia covariants we then see that they are given by
\[J_{1}(F) = 3x^2 + 6xy + 2y^2, J_{2}(F) = 3x^2 - y^2, J_{3}(F) = -6x^2 - 6xy - y^2, \]
whence
\[\M_{\J_1} = \begin{pmatrix} 6 & 4 \\ -6 & -6 \end{pmatrix}, \M_{\J_2} = \begin{pmatrix} 0 & -2 \\ -6 & 0 \end{pmatrix}, \M_{\J_3} = \begin{pmatrix} -6 & -2 \\ 12 & 6 \end{pmatrix}.\]
The Hessian covariant of $F$ is given by
\[q_F(x,y) = 3x^2 + 3xy + y^2 \]
and 
\[\M_{q_F} = \begin{pmatrix} 3 & 2 \\ -6 & -3 \end{pmatrix}.\]
It thus follows that
\[\T_1 = \begin{pmatrix} 1 & 0 \\ -3 & -1 \end{pmatrix}, \T_2 = \begin{pmatrix} -2 & -1 \\ 3 & 2 \end{pmatrix}, \T_3 = \begin{pmatrix} 1 & 1 \\ 0 & -1 \end{pmatrix}.\]
A quick calculation then shows that $\T_1, \T_2, \T_3$ fix $F$ by substitution. Similarly, one can check that
\[\N_{q_F} = \begin{pmatrix} -2 & -1 \\ 3 & 1 \end{pmatrix}\]
and its square both fix $F$ via substitution. This completes the proof. 
\end{proof}
We remark that the requirement for the $x^3$-coefficient of $F$ be non-zero is merely in place because of how the Julia covariants are defined. Indeed the statement holds for all binary cubic forms with non-zero discriminant, because the Julia covariants are covariants. \\

We may now prove Theorem \ref{BCFMT}. 

\begin{proof}[Proof of Theorem \ref{BCFMT}] Put $\Aut^\ast F$ for the subset of $\Aut_\bC^\ast F$ defined over the rationals. We note that the natural map
\[\Aut^\ast F \rightarrow \Aut F\]
is an isomorphism, since $\bQ$ does not contain any non-trivial cube roots of unity. Thus, the elements of $\Aut F$ must come from the set $\{I_{2 \times 2}, \T_{\theta_1}, \T_{\theta_2}, \T_{\theta_3}, \N_{q_F}, \N_{q_F}^2\}$. \\ \\
Observe that $\M_{q_F} \in \GL_2(\bQ)$ for all binary cubic forms with integer coefficients. From here it is plain that $\T_\theta$ can be in $\GL_2(\bQ)$ only if $\M_{\J_\theta}$ has rational coefficients, and from (\ref{Julia cov}) we see that this can only occur when $\theta$ is a rational. Therefore if $F$ is irreducible, then $\T_\theta$ does not lie in $\Aut F$. By examining the explicit formula in (\ref{Hooley matrix}), it follows that $\N_{q_F} \in \GL_2(\bQ)$ only when $-3\Delta(q_F)$ is a square, which is equivalent to $\Delta(F)$ being a square. Thus, when $\Delta(F)$ is not a square and $F$ is irreducible, $\Aut F$ contains just the identity matrix. \\ \\
When $F$ is reducible, say $\theta$ is a rational root of $F(x,1)$, we see that $\T_\theta$ does indeed lie in $\GL_2(\bQ)$. An elementary calculation shows that if $F(x,1)$ has a unique rational root then $\Delta(F)$ is not a square, and thus $\N_{q_F} \not \in \GL_2(\bQ)$. Therefore, $\T_\theta$ is the only non-trivial element of $\Aut F$. Finally, if $F(x,1)$ has three rational roots, it is obvious from the definition of the discriminant that $\Delta(F)$ is a square and hence $\T_{\theta_i}, i = 1,2,3$ and $\N_{q_F}, \N_{q_F}^2$ are all rational. 
\end{proof}

\section{Binary quartic forms}
\label{quartics}

Suppose
\[F(x,y) = a_4 x^4 + a_3 x^3 y + a_2 x^2 y^2 + a_1 xy^3 + a_0 y^4\]
is a binary quartic form with integer coefficients and non-zero discriminant. For a binary quadratic form $f(x,y) = ax^2 + bxy + cy^2$ with real coefficients and non-zero discriminant, put
\begin{equation} \label{U2} U_f = \frac{1}{\sqrt{|\Delta(f)|}} \M_{f}. \end{equation}
We say that a binary quadratic form $f$ with complex coefficients is \emph{rationally good} if it is proportional over $\bC$ to a quadratic form $g$ with integer coefficients and $|\Delta(g)|$ the square of an integer. Otherwise, we say that $f$ is \emph{rationally bad}. \\ \\
Binary quartic forms have a degree 6 covariant given by
\begin{align} \label{sextic cov} 
F_6(x,y) & = (a_3^3 + 8a_4^2 a_1 - 4a_4 a_3 a_2) x^6 + 2(16a_4^2 a_0 + 2 a_4 a_3 a_1 - 4a_4 a_2^2 + a_3^2 a_2)x^5y \\
& + 5(8a_4 a_3 a_0 + a_3^2 a_1 - 4a_4 a_2 a_1)x^4y^2 + 20(a_3^2 a_0 - a_4 a_1^2) x^3 y^3 \notag \\
& - 5(8a_4 a_1 a_0 + a_3 a_1^2 - 4 a_3 a_2 a_0)x^2 y^4 - 2(16a_4 a_0^2 + 2 a_3 a_1 a_0 - 4 a_2^2 a_0 + a_2 a_1^2)xy^5   \notag \\ 
& - (a_1^3 + 8a_3 a_0^2 - 4 a_2 a_1 a_0)y^6. \notag 
\end{align}
We call a quadratic form divisor $f$ of $F_6$ \emph{significant} if the quartic form $G = F_6/f$ satisfies $J(G) = 0$. \\ 

It turns out that the covariant $F_6$ and its significant factors controls the behaviour of $\Aut F$. We then have the following theorem:

\begin{theorem} \label{BQFMT} Let $F$ be a binary quartic form with integer coefficients and non-zero discriminant. 
\begin{enumerate}
\item $\Aut F = \{\pm I_{2 \times 2}\}$ if and only if $F_6$ does not have any real rationally good significant quadratic factors. 
\item $\Aut F$ is generated  by $U_f \in \GL_2(\bQ)$ and $-I_{2 \times 2}$ if and only if $F_6$ has a unique real rationally good significant factor $f$. In this case $\Aut F$ is isomorphic to $\C_2 \times \C_2$ or $\C_4$. 
\item 
\[\Aut F = \{\pm I_{2 \times 2}, \pm U_{f_1}, \pm U_{f_2}, \pm U_{f_3}\} \cong \D_4\]
if and only if $F_6$ can be written as $ F_6 = f_1 f_2 f_3$ where $f_i$ is a real rationally good significant factor of $F_6$ for $i = 1,2,3$. 
\end{enumerate}
\end{theorem}

We remark that the sextic covariant $F_6$ of a binary quartic form $F$ is always a \emph{Klein form}; see Lemma \ref{sextic Klein}. This fact does not appear to be well-known. Given the significance of Klein forms in problems involving the super-elliptic equation (see \cite{BeDa}), this phenomenon may be of independent interest.

\subsection{Binary sextic Klein forms and significant quadratic factors}

There is a simple characterization of the elements in $\Aut F$ in terms of \emph{significant} quadratic factors of the sextic covariant $F_6$ given in (\ref{sextic cov}). \\ \\
A degree six binary form
 \[G(x,y) = g_6 x^6 + g_5 x^5 y + g_4 x^4 y^2 + g_3 x^3 y^3 + g_2 x^2 y^4 + g_1 xy^5 + g_0 y^6\] 
is said to be a \emph{Klein form} if its coefficients satisfy the following quadratic equations (see \cite{BeDa}):
\begin{align} \label{Klein} & 10 g_6 g_2 - 5 g_5 g_3 + 2 g_4^2 = 0 \\
& 25 g_6 g_1 - 5 g_5 g_2 + g_3 g_4 = 0 \notag \\
& 50 g_6 g_0 - 2 g_2 g_4 + g_3^2 = 0. \notag
\end{align}
Moreover it is known that all binary sextic Klein forms with complex coefficients and non-zero discriminant are $\GL_2(\bC)$-equivalent to each other, a fact already known to Klein \cite{Kle}. \\ 

We have the following fact, which appears to be new:

\begin{lemma} \label{sextic Klein} Let $F$ be a binary quartic form with complex coefficients and non-zero discriminant. Then its sextic covariant $F_6$, given in (\ref{sextic cov}), is a Klein form with non-zero discriminant. 
\end{lemma}

\begin{proof} Since all binary quartic forms with non-zero discriminant are equivalent to a form of the shape $x^4 + A x^2 y^2 + y^4$ for some complex number $A$, it suffices to verify that the sextic covariant of $F(x,y) = x^4 + Ax^2 y^2 + y^4$ is a Klein form. A quick calculation shows that $F_6$ is proportional over $\bC$ to 
\[G(x,y) = xy(x^4 - y^4),\]
which is independent of $A$. We then see that $\Delta(G) \ne 0$ and that the coefficients of $G$ satisfy the quadratic equations in (\ref{Klein}). 
\end{proof} 

By Lemma \ref{sextic Klein}, the deduction of Theorem \ref{BQFMT} from Proposition \ref{real quartic auto} will follow from the following lemmas. 

\begin{lemma} \label{sig facts} Let $G$ be a sextic Klein form with non-zero discriminant. Then $G$ can be written as $G = G_1 G_2 G_3$, where each $G_i$ is a significant quadratic factor of $G$, in only one way up to permutation of the factors and up to homothety over $\bC$. 
\end{lemma} 

\begin{proof} Since all binary sextic Klein forms with non-zero discriminant lie in a single $\GL_2(\bC)$-orbit, it suffices to prove Lemma \ref{sig facts} for just the Klein form $G(x,y) = xy(x^4 - y^4)$. We see that factoring $G$ as $G = G_1 G_2 G_3$, with 
\[G_1 = xy, G_2 = x^2 - y^2, G_3 = x^2 + y^2\]
that each $G_i$ is a significant factor of $G$; that is, the quartic form $\G_1 = G/G_1 = x^4 - y^4$ satisfies $J(\G_1) = 0$, and similarly for $\G_i = G/G_i$ for $i = 2,3$. \\ \\
Now pick another quadratic factor of $G$, say $V(x,y) = x(x+y)$. Then $\V = G/V = y(x-y)(x^2 +  y^2)$ has
\[J(\V) = 72(0)(-1)(-1) + 9(1)(-1)(1) - 27(0)(1)^2 - 27(-1)(1)^2 - 2 (-1)^3 = 20.\]
A similar calculation shows that for any other quadratic factor $V$ distinct from $G_1, G_2, G_3$, that $J(G/V) \ne 0$, whence $V$ is not a significant factor of $G$. 
\end{proof}

\subsection{$\Aut_\bC^\ast F$ for binary quartic forms}
\label{real automorphisms}

In this section we aim to show that $\Aut_\bC^\ast (F)$ is determined explicitly by certain quadratic covariants of a binary quartic form $F$, called the \emph{Cremona covariants}, which are significant divisors of $F_6$. Let 
\[F(x,y) = a_4 x^4 + a_3 x^3 y + a_2 x^2 y^2 + a_1 xy^3 + a_0 y^4\]
be a binary quartic form. We shall assume, by applying a $\GL_2(\bZ)$-action if necessary, that $a_4 \ne 0$. Unlike the cubic case, there are no rational quadratic covariants for binary quartic forms. However, there are three \emph{irrational} quadratic covariants discovered by Cremona \cite{Cre}. These covariants can be given explicitly in terms of the roots of $F(x,1)$. Define $\chi(F)$ to be the number of real roots of $F(x,1)$. We will then label the roots $\theta_i$, $i = 1,2,3,4$ of $F(x,1)$ as in \cite{BE}: 
\begin{equation} \label{Bean label} \begin{cases} \theta_1 > \theta_2 > \theta_3 > \theta_4, & \text{if } \chi(F) = 4, \\ \\
\theta_1 > \theta_2, \theta_3 = \ol{\theta_4}, \Im(\theta_3) > 0, & \text{if } \chi(F) = 2, \\ \\
\theta_1 = \ol{\theta_2}, \theta_3 = \ol{\theta_4}, \Im(\theta_1) > 0 , \Im(\theta_3) < 0, & \text{if } \chi(F) = 0. \end{cases}
\end{equation}
Here $\Im(z)$ refers to the imaginary part of the complex number $z$. Put
\begin{equation} \label{Auto coeff} A_1 = a_4(\theta_1 + \theta_2 - \theta_3 - \theta_4), B_1 = 2a_4(\theta_3 \theta_4 - \theta_1 \theta_2), C_1 = a_4(\theta_1 \theta_2(\theta_3 + \theta_4) - \theta_3 \theta_4(\theta_1 + \theta_2)), \end{equation}
\[A_2 = a_4(\theta_1 + \theta_3 - \theta_2 - \theta_4), B_2 = 2a_4(\theta_2 \theta_4 - \theta_1 \theta_3), C_2 = a_4(\theta_1 \theta_3(\theta_2 + \theta_4) - \theta_2 \theta_4(\theta_1 + \theta_3)),\]
and
\[A_3 = a_4(\theta_1 + \theta_4 - \theta_2 - \theta_3), B_3 = 2a_4(\theta_2 \theta_3 - \theta_1 \theta_4), C_3 = a_4(\theta_1 \theta_4(\theta_2 + \theta_3) - \theta_2 \theta_3(\theta_1 + \theta_4))\]
and define the $i$-th \emph{Cremona covariant} to be
\begin{equation} \label{Cremona cov} \fC_i(x,y) = A_i x^2 + B_i xy + C_i y^2, i = 1,2,3.\end{equation}
Put
\begin{equation} \label{Cremona disc} D_i = \Delta(\fC_i) \text{ for } i = 1,2,3. \end{equation}
One checks that the $D_i$'s satisfy
\begin{equation} \label{Xi} D_1 = 4a_4^2(\theta_1 - \theta_3)(\theta_1 - \theta_4)(\theta_2 - \theta_3)(\theta_2 - \theta_4), \end{equation}
\[D_2 = 4a_4^2(\theta_1 - \theta_2)(\theta_1 - \theta_4)(\theta_3 - \theta_2)(\theta_3 - \theta_4),\]
and
\[D_3 = 4a_4^2(\theta_1 - \theta_2)(\theta_1 - \theta_3)(\theta_4 - \theta_2)(\theta_4 - \theta_3).\]
We note that (\ref{Xi}) implies that $D_i \ne 0$ for $i = 1,2,3$ whenever $\Delta(F) \ne 0$. In \cite{Cre}, Cremona showed that the Cremona covariants $\fC_i$ satisfies 
\begin{equation} \label{F6 factor} F_6(x,y) = \fC_1 (x,y) \fC_2 (x,y) \fC_3 (x,y).
\end{equation} 

Put 
\[\U_i = \frac{1}{\sqrt{D_i}} \M_{\fC_i}, i = 1,2,3.\]
We have the following proposition:

\begin{proposition} \label{quartic complex} Let $F$ be a binary quartic form with complex coefficients and non-zero discriminant. Suppose that the $x^4$-coefficient of $F$ is non-zero and that $I(F) J(F) \ne 0$. Then a set of representatives of $\Aut_\bC^\ast (F)$ in $\GL_2(\bC)$ is given by
\[\{I_{2 \times 2}, \U_1, \U_2, \U_3\}.\]
Moreover, for each $i = 1,2,3$ we have $F_{\U_i} = F$ with respect to the action (\ref{sub action}). 
\end{proposition} 

\begin{proof} By Proposition \ref{PGL auto prop} and its proof, it follows that $\Aut_\bC^\ast (F) \cong \C_2 \times \C_2$. Therefore, it suffices to check that $\M_{\fC_i} \in \Aut_\bC^\ast(F)$ for each $i = 1,2,3$. Let us consider the action of $\M_{\fC_1}$ on $\theta_1$, via the action in (\ref{U action}). We have
\[U_1 : \theta_1 \mapsto \frac{B_1 \theta_1 + 2C_1}{-2A_1 \theta_1 - B_1}.\]
Expanding using (\ref{Auto coeff}), we obtain
\[\frac{B_1 \theta_1 + 2C_1}{-2A_1 \theta_1 - B_1} = \frac{-2 \theta_2 (\theta_1 - \theta_3)(\theta_1 - \theta_4)}{-2(\theta_1 - \theta_3)(\theta_1 - \theta_4)} = \theta_2. \]
Next we see that
\[\frac{B_1 \theta_3 + 2C_1}{-2A_1 \theta_3 - B_1} = \frac{2 \theta_4(\theta_3 - \theta_1)(\theta_3 - \theta_2)}{2(\theta_3 - \theta_1)(\theta_3 - \theta_2) } = \theta_4.\]
A similar calculation shows that $U_1$ sends $\theta_2$ to $\theta_1$ and $\theta_4$ to $\theta_3$. This shows that $\M_{\fC_1}$ permutes the roots of $F$. A similar calculation shows that $\M_{\fC_2}, \M_{\fC_2}$ similarly permute the roots of $F(x,1)$. \\ \\
To confirm that $F_{\U_1} = F$ say, we further need to check that $\U_1$ fixes the leading coefficient of $F$. This is equivalent to checking that
\begin{equation} \label{two sides} \frac{1}{D_1^2} \left(a_4 B_1^4 + a_3 B_1^3 (-2A_1) + a_2 B_1^2 (-2A_1)^2 + a_1 B_1 (-2A_1)^3 + a_0 (-2A_1)^4\right) = a_4. \end{equation}
Using the fact that $a_4 \ne 0$ and the Vieta relations
\[\frac{a_3}{a_4} = -(\theta_1 + \theta_2 + \theta_3 + \theta_4),\]
\[\frac{a_2}{a_4} = \theta_1 \theta_2 + \theta_1 \theta_3 + \theta_1 \theta_4 + \theta_2 \theta_3 + \theta_2 \theta_4 + \theta_3 \theta_4,\]
\[\frac{a_1}{a_4} = - (\theta_1 \theta_2 \theta_3 + \theta_1 \theta_2 \theta_4 + \theta_1 \theta_3 \theta_4 + \theta_2 \theta_3 \theta_4),\]
and
\[\frac{a_0}{a_4} = \theta_1 \theta_2 \theta_3 \theta_4,\]
we see that (\ref{two sides}) is equivalent to checking that
\[(\theta_3 \theta_4 - \theta_1 \theta_2)^4 + (\theta_1 + \theta_2 + \theta_3 + \theta_4)(\theta_3 \theta_4 - \theta_1 \theta_2)^3(\theta_1 + \theta_2 - \theta_3 - \theta_4) \]
\[+ ( \theta_1 \theta_2 + \theta_1 \theta_3 + \theta_1 \theta_4 + \theta_2 \theta_3 + \theta_2 \theta_4 + \theta_3 \theta_4) (\theta_3 \theta_4 - \theta_1 \theta_2)^2(\theta_1 + \theta_2 - \theta_3 - \theta_4)^2\]
\[+(\theta_1 \theta_2 \theta_3 + \theta_1 \theta_2 \theta_4 + \theta_1 \theta_3 \theta_4 + \theta_2 \theta_3 \theta_4)(\theta_3 \theta_4 - \theta_1 \theta_2)(\theta_1 + \theta_2 - \theta_3 - \theta_4)^3\]
\[+ \theta_1 \theta_2 \theta_3 \theta_4 (\theta_1 + \theta_2 - \theta_3 - \theta_4)^4\]
is equal to
\[(\theta_1 - \theta_3)^2(\theta_1 - \theta_4)^2(\theta_2 - \theta_3)^2(\theta_2 - \theta_4)^2.\]
This can be done using any standard computer algebra package (in particular, we used Sage). Thus $\U_1 \in \Aut_\bC F$. The verification that $\U_2, \U_3 \in \Aut_\bC F$ follows similarly.  \end{proof}
Again we remark that the requirement for the $x^4$-coefficient to be non-zero is not essential, since the $\fC_i$'s are covariants. There is a more intrinsic way to define the Cremona covariants in terms of the \emph{cubic resolvent} of $F$ and the Hessian covariant; see \cite{Cre}. \\

The following lemma shows that the Cremona covariants $\fC_i$ are precisely the significant factors of $F_6$. 

\begin{lemma} \label{Cre sig} Let $F$ be a binary quartic form with complex coefficients and non-zero discriminant. Then for each $i = 1,2,3$, the Cremona covariant $\fC_i$ of $F$ is a significant factor of the sextic covariant $F_6$. 
\end{lemma}

\begin{proof} Recall that each binary quartic form $F$ with complex coefficients and non-zero discriminant is equivalent to $F_A = x^4 + Ax^2 y^2 + y^4$ for some complex number $A$, and that the Cremona covariants of $F_A$ are proportional to $xy, x^2 - y^2, x^2 + y^2$. Lemma \ref{Cre sig} then follows from Lemma \ref{sig facts}. 
\end{proof}

\subsection{$\Aut_\bR F$ for real binary quartic forms}
\label{real quartic autos}

Even though we are primarily interested in $\Aut F$, which is defined to be the set of $T \in \GL_2(\bQ)$ which fixes $F$ via the action (\ref{sub action}), it will be convenient to first consider the larger group $\Aut_\bR F$. It is clear that $\Aut F \subset \Aut_\bR F$. Proposition \ref{quartic complex} shows that the matrices $\U_{i}, i = 1,2,3$ are in $\Aut_\bC F$, it thus remains to check that when it is possible that $\U_{i} \in \GL_2(\bR)$, possibly up to multiplying by a 4-th root of unity. We have the following proposition:

\begin{proposition} \label{real quartic auto} Let $F$ be a binary quartic form with real coefficients and non-zero discriminant. Then $\Aut_\bR F$ is given by: 
\[\begin{cases} \left\{\pm I_{2 \times 2}, \pm U_{\fC_1}, \pm U_{\fC_2}, \pm U_{\fC_3} \right\} & \text{if } \chi(F) = 4, \\
\left\{\pm I_{2 \times 2}, \pm U_{\fC_1} \right\} & \text{if } \chi(F) = 2, \\
\left\{\pm I_{2 \times 2}, \pm U_{\fC_1}, \pm \mu_4 U_{\fC_2}, \pm \mu_4 U_{\fC_3} \right\} & \text{if } \chi(F) = 0.
\end{cases} 
\]
\end{proposition} 

\begin{proof} When $\chi(F) = 4$, it is obvious that each $\fC_i$ is real and thus $\U_{i}$ is real as long as $\Delta(\fC_i)$ is positive. This holds for $i = 1,3$ but $D_2 < 0$, whence $\sqrt{D_2} = \mu_4 \sqrt{|D_2|}$. Therefore $\mu_4 \U_2  = U_{\fC_2} \in \Aut_\bR F$, as desired. \\

When $\chi(F) = 2$, from (\ref{Bean label}) and (\ref{Auto coeff}) we see that $\fC_1$ is real with positive discriminant, while $\fC_2, \fC_3$ are neither real nor purely imaginary. Moreover, neither can be proportional over $\bC$ to a real form. To see this, observe that from an examination of (\ref{Bean label}) we see that $\fC_2, \fC_3$ have coefficients which are conjugate in $\bC$. Thus, $\fC_2$ is proportional to a real form if and only if $\fC_3$ is proportional to a real form; and moreover, they must be proportional to each other. This implies that $F_6 = \fC_1 \fC_2 \fC_3$ is a singular form, which by Lemma \ref{quartic conj} and the proof of Lemma \ref{sextic Klein} shows that $F$ itself must have vanishing discriminant. \\ 


When $\chi(F) = 0$, we see that $\fC_1$ is real with positive discriminant while $\fC_2, \fC_3$ have coefficients which are purely imaginary, and thus multiplying by $i = \mu_4$ turns them into real quadratic forms. It thus follows that in each case, $\Aut_\bR F$ contains the sets given in the proposition. \\ \\
It remains to check that $\Aut_\bR F$ cannot be any larger in the cases when $I(F)$ or $J(F)$ vanishes. When $I(F) = 0$ this easily follows since $\Aut_\bR F$, being a finite subgroup of $\GL_2(\bR)$, cannot contain a copy of $\A_4$. When $J(F) = 0$ we see that the preimage of an order $4$ element in $\Aut_\bC^\ast(F)$ is necessarily an element $T$ of order $8$ in $\GL_2(\bR)$. Since all elements of order $8$ in $\GL_2(\bR)$ are conjugate, we may then assume $T \in \text{SO}_2(\bR)$. But then by letting $T$ permute the roots of a binary quartic form we see that $T$ necessarily sends $F$ to $-F$, hence $T \not \in \Aut_\bR F$. 
\end{proof}

We may now give a proof of Theorem \ref{BQFMT}. 

\begin{proof}[Proof of Theorem \ref{BQFMT}] By Proposition \ref{real quartic auto}, the potential non-trivial elements of $\Aut F$ are given explicitly in terms of the Cremona covariants. For each $\ep U_{\fC_i} \in \Aut_\bR F$, where $\ep \in \{1, \mu_4\}$, we have that $\ep U_{\fC_i} \in \Aut F$ only if $\fC_i$ is proportional over $\bC$ to an integral binary quadratic form $f_i$. In this case we have
\[\ep U_{\fC_i} = \frac{1}{\sqrt{|\Delta(f_i)|}} \begin{pmatrix} f_1 & 2 f_0 \\ -2 f_2 & -f_1 \end{pmatrix}.\]
Then we see that $\ep U_{\fC_i} \in \Aut F$ only if $|\Delta(f_i)|$ is a square; that is, $\fC_i$ is rationally good. We then see that these conditions are also sufficient for $\ep U_{\fC_i} \in \GL_2(\bQ)$. \\ \\
Therefore, when $F_6$ has no real quadratic rationally good significant factors, $\Aut_\bR F = \{\pm I_{2 \times 2}\}$. When $F_6$ has exactly one real quadratic rationally good significant factor $f$, it can have positive or negative discriminant, which will determine the order of $U_{f}$ in $\GL_2(\bR)$. If the $\Delta(f) < 0$ then $U_f$ will have order 4 and $\Aut_\bR F \cong \C_4$, and if $\Delta(f) > 0$ then $U_f$ has order 2 and $\Aut_\bR F \cong \C_2 \times \C_2$. Finally, if $F_6$ has three real quadratic rationally good significant factors $f_i, i = 1,2,3$ then $f_i$ is proportional over $\bR$ to $\ep U_{\fC_i}$ for $i = 1,2,3$ and so $\Aut F$ is as given by Proposition \ref{real quartic auto}. 
\end{proof}

\section{Cubic and quartic surfaces defined by binary forms and Theorem \ref{line thm}}
\label{surfaces}

In this section we apply our theorems characterizing the automorphism groups of binary cubic and quartic forms $F$ to study lines on the surface $X_F$ given in (\ref{surface def}), and to prove Theorem \ref{line thm}. \\ 

For a binary form $F$, denote by $\B_F$ the set of projective roots of $F$ in $\bP^1(\bC)$. Let $F_1, F_2$ be two binary forms with complex coefficients and non-zero discriminant. Put $\G(F_1, F_2)$ for the set of elements in $\PGL_2(\bC)$ which map $\B_{F_1}$ to $\B_{F_2}$. When $\deg F_1 \ne \deg F_2$, it is obvious that $\G(F_1, F_2)$ is empty. When $\deg F_1 = \deg F_2 = 3$, the set $\G(F_1, F_2)$ always consists of six elements since $\PGL_2(\bC)$ is $3$-transitive on $\bP^1(\bC)$. When $\deg F_1 = \deg F_2 = 4$, the cardinality of $\G(F_1, F_2)$ can be $0, 4, 8, 12$. Put
\begin{equation} \label{number of autos} \upsilon_{F_1, F_2} = \# \G(F_1, F_2) \end{equation}
and
\[\upsilon_F = \upsilon_{F, F}.\]

Consider the surface $X_{F_1, F_2}$ defined by
\[F_1(x_1, x_2) - F_2(x_3, x_4) = 0.\]
We then have the following, which is Theorem 3.1 in \cite{BSa}:

\begin{proposition}[Theorem 3.1 in \cite{BSa}] \label{BSa MT} Let $F_1, F_2$ be two binary forms with non-zero discriminant and equal degree $d$. Then the number of lines on the surface $X_{F_1, F_2}$ is equal to $d(d + \upsilon_{F_1, F_2})$. 
\end{proposition}

By Proposition \ref{BSa MT}, the number of lines on a surface $X_F$ with $F$ a binary form with non-zero discriminant is completely determined by $\Aut_\bC^\ast F$. To prove Theorem \ref{line thm}, however, we shall need the following refinement of Proposition \ref{BSa MT}, which is contained in the proof of Theorem 3.1 in \cite{BSa}. \\ \\
We shall denote a projective point in $\bP^3$ by $[x_1 : x_2 : x_3 : x_4]$. Let $F$ be a binary form of degree $d$ with complex coefficients and non-zero discriminant. By applying a $\GL_2(\bC)$ transformation, we may assume that the leading coefficient of $F$ is non-zero.

\begin{lemma} \label{line cat}  Let $\psi_1, \cdots, \psi_d$ denote the roots of $F(x,1)$. Then all lines on the surface $X_{F}$ are in exactly one of the following two categories:
\begin{itemize}
\item[(a)] (Root lines) $L = \{[s \psi_j : s : t \psi_k : t] \in \bP^3: s,t \in \bC \}$ for some $1 \leq j,k \leq d$, or
\item[(b)] (Automorphism lines) There exists an automorphism 
\[T = \begin{pmatrix} t_1 & t_2 \\ t_3 & t_4 \end{pmatrix} \in \Aut_\bC F\]
such that
\[L = \{[u : v : t_1 u + t_2 v : t_3 u + t_4 v] \in \bP^3 : u, v \in \bC\}.\]
\end{itemize}
\end{lemma}

\begin{proof} The proof of Lemma \ref{line cat} follows from the proof of Theorem 3.1 in \cite{BSa}. The second part of Lemma \ref{line cat} was done by Heath-Brown in \cite{HB1}. 
\end{proof}

\subsection{Cubic surfaces}

We shall state a more precise version of part (a) of Theorem \ref{line thm}. For a given binary form $F$ with integer coefficients and non-zero discriminant, write $K$ for the field of smallest degree for which all lines contained in the surface $X_F$ defined by (\ref{surface def}) is defined over $K$. We put $\bF$ for the splitting field of $F$, and we shall denote by $\mu_3$ for a primitive third root of unity. We will prove the following for cubic surfaces $X_F$ defined (\ref{surface def}) and a binary cubic form $F(x,y)$:

\begin{proposition} \label{cubic surface} Let $F$ be a binary cubic form with integer coefficients and non-zero discriminant. Then $K = \bF(\mu_3)$. 
\end{proposition}

\begin{proof} From Proposition \ref{BSa MT} we see that the root lines of $X_F$ are defined over $\bF$. For the automorphism lines, we see that $q_F$, and hence $\M_{q_F}$, is defined over $\bQ$ whenever $F$ has integer coefficients, and that $\M_{\J_\theta}$ is defined over $\bF$. Hence, $\T_{\theta}$ is defined over $\bF$. Note that the definition $\N_{q_F}$ and $\N_{q_F}^2$ involves the term $\sqrt{-3\Delta(F)}$, which may not lie in $\bF$. We note however that $\sqrt{-3} \in \bQ(\mu_3)$ and $\sqrt{\Delta(F)} \in \bF$, since it is the product of the differences of the roots of $F(x,1)$. Therefore, all of the automorphisms of $F$ are defined over $\bF(\mu_3)$ and hence all of the automorphism lines are defined over $\bF(\mu_3)$. 
\end{proof}

Finally, it is clear that Proposition \ref{cubic surface} implies the cubic case of Theorem \ref{line thm}, since $[\bF(\mu_3) : \bQ] \leq 12$.  

\subsection{Quartic surfaces} 

Let $F$ be a binary quartic form with integer coefficients and non-zero discriminant. By Propositions \ref{PGL auto prop} and \ref{BSa MT}, the surface $X_F$ contains either $32, 48,$ or $64$ lines depending on whether $I,J$ vanish. Put $\sigma_F$ for the number of lines contained in the surface $X_F$ given in (\ref{surface def}). We then have the following proposition:

\begin{proposition} Let $F$ be a binary quartic form with integer coefficients and non-zero discriminant. \begin{enumerate}
\item If $I(F), J(F)$ are both non-zero, then $\sigma_F = 32$. Moreover, $[K:\bQ] \leq 48$ with equality holding whenever $\Gal F \cong \S_4$ and $\Delta(F)$ is not the negative of a square integer. 
\item If $I(F) = 0$, then  $\sigma_F = 64$.
\item If $J(F) = 0$, then $\sigma_F = 48$. 
\end{enumerate}
\end{proposition}

\begin{proof} By Propositions \ref{PGL auto prop}, the size of $\Aut_\bC^\ast F$ is $4, 8, 12$ respectively when $I(F), J(F)$ are both non-zero, when $J(F) = 0$ and $I(F) \ne 0$, and $I(F) = 0$ with $J(F) \ne 0$. Thus, by Proposition \ref{BSa MT}, we have
\[\sigma_F = \begin{cases} 4(4 + 4) = 32 & \text{if } I(F) J(F) \ne 0, \\
4(4 + 8) = 48 & \text{if } J(F) = 0, I(F) \ne 0, \\
4(4+ 12) = 64 & \text{if } I(F) = 0, J(F) \ne 0. \end{cases}\]
We now prove part (1) of the proposition. The treatment of the root lines is the same as the cubic case, and it is clear that the root lines are defined over $\bF$. For the automorphism lines, the lines corresponding to $\M_{\fC_1}, \M_{\fC_2}, \M_{\fC_3}$ are defined over $\bF$ by (\ref{Auto coeff}). The remaining automorphism lines are defined after adjoining $\mu_4$ to $\bF$, whence $K = \bF(\mu_4)$. Moreover $\bF$ is at most a degree $24$ extension over $\bQ$, thus $[K : \bQ] \leq 48$. Observe that equality holds only when $[\bF : \bQ] = 24$, which implies that $\Gal F \cong \S_4$, and that $\bF(\mu_4) \ne \bF$. This condition is equivalent to $\mu_4 \not \in \bF$, and an elementary exercise in Galois theory yields that this happens if and only if $\Delta(F)$ is not the negative of a square integer. 
 \end{proof}

\end{document}